\documentclass{amsart}

\pdfoutput=1

\usepackage[utf8]{inputenc}
\usepackage{amsmath}
\usepackage{amssymb}
\usepackage{amsthm}
\usepackage{color}
\usepackage{graphicx}
\usepackage{enumerate}
\usepackage{mathrsfs}
\usepackage{hyperref}

\usepackage{caption}
\usepackage{subcaption}

\usepackage[T1]{fontenc}

\usepackage{biblatex}
\newtheorem{theorem}{Theorem}[section]

\newtheorem{prop}[theorem]{Proposition}

\newtheorem{lemma}[theorem]{Lemma}

\newtheorem{cor}[theorem]{Corollary}

\newtheorem*{theorem*}{Theorem}

\theoremstyle{remark}

\usepackage[normalem]{ulem}  
\definecolor{jimred}{rgb}{0.75,0.25,0}
\definecolor{jimblue}{rgb}{0.1,0.55,0.1}

\newcommand{\jout}{\bgroup\markoverwith{\textcolor{jimred}{\rule[0.5ex]{2pt}{1pt}}}\ULon}

\title[Pseudo-$F_4$ is isomorphic to $F_4$]{Pseudo-$\boldsymbol{F_4}$ is isomorphic to $\boldsymbol{F_4}$}
\author{James Belk}
\address{School of Mathematics and Statistics, University of Glasgow, University Place, Glasgow, G12~8QQ, Scotland.}
\email{\href{mailto:jim.belk@glasgow.ac.uk}{jim.belk@glasgow.ac.uk}}
\thanks{The first author has been partially supported by the National Science Foundation under Grant No.~DMS-1854367 during the creation of this paper.}

\author{Liam Stott}
\address{School of Mathematics and Statistics, University of St Andrews, North Haugh, St Andrews, KY16 9SS, Scotland.}
\email{\href{mailto:lks4@st-andrews.ac.uk}{lks4@st-andrews.ac.uk}}
\thanks{The second author has been partially funded by the EPSRC through the University of St Andrews doctoral training program during the creation of this paper.}

\date{} 
\newcommand{\nbumps}{\mathfrak{C}_n}

\DeclareMathOperator{\supt}{supt}
\DeclareMathOperator{\src}{src}
\DeclareMathOperator{\dest}{dest}
\DeclareMathOperator{\Homeo}{Homeo}

\definecolor{myred}{rgb}{0.85,0,0}
\definecolor{myblue}{rgb}{0,0,0.75}

\begin{document}

\begin{abstract}
We prove that the ``pseudo-$F_4$'' group is isomorphic to~$F_4$, answering a question of Brin.  Both of these groups can be described as fast groups of homeomorphisms of the interval generated by bumps, as introduced by Bleak, Brin, Kassabov, Moore, and Zaremsky.  The proof uses a representation of fast groups as Guba--Sapir diagram groups in order to leverage known results on isomorphisms of diagram groups.
\end{abstract}

\maketitle

\section{Introduction}

If $f$ is a homeomorphism of $[0,1]$, an \textbf{orbital} for $f$ is a maximal open interval on which $f$ has no fixed points.  A homeomorphism $f\colon [0,1]\to [0,1]$ is called a \textbf{bump} if it is orientation-preserving and has exactly one orbital.  Such a bump is \textbf{positive} if $f(x)\geq x$ for all $x\in[0,1]$.

In a 2019 paper \cite{fast}, Bleak, Brin, Kassabov, Moore, and Zaremsky proved that under a certain condition, the isomorphism type of a group generated by finitely many bumps depends only on the way in which the orbitals overlap. They further extend this condition, and their result, to finite sets of orientation-preserving homeomorphisms with finitely many orbitals. They refer to this condition as being ``fast'', and the resulting groups are ``fast groups''.  Using fast groups, Bleak, Brin, and Moore went on to construct a transfinite sequence $\{G_\xi\}_{\xi<\varepsilon_0}$ of elementary amenable subgroups of Thompson's group~$F$ that is linearly ordered by the embeddability relation and has the property that for any $n\in\mathbb{N}$ and any ordinal $0<\alpha<\varepsilon_0$, the group $G_\xi$ for $\xi=\omega^{(\omega^\alpha)\cdot (2^n)}$ is elementary amenable of class $\omega\cdot \alpha + n + 2$ \cite{EA}.

The class of fast groups generated by bumps is defined as follows. If $f$ is a positive bump with orbital~$(a,b)$, a \textbf{pair of feet} for $f$ are intervals of the form $L= (a,c)$ and $R=(f(c),b)$, where $c$ is some chosen point in $(a,b)$.  A set $f_1,\ldots,f_n$ of positive bumps is \textbf{fast} if we can choose a pair of feet $L_i,R_i$ for each $f_i$ so that the intervals $L_1,R_1,\ldots,L_n,R_n$ are pairwise disjoint.  In this case, we refer to the group $\langle f_1,\ldots,f_n\rangle$ as a \textbf{fast group}.

The following is a special case of the remarkable theorem proved by Bleak, Brin, Kassabov, Moore, and Zaremsky.  The proof resembles that of the classical ping-pong lemma for free groups.

\begin{theorem}[Ping-Pong for Fast Groups \cite{fast}]\label{thm:PingPong}
Let\/ $\{f_1,\ldots, f_n\}$ be a fast set of positive bumps with pairwise disjoint feet\/ $\{L_i,R_i\}_{i=1}^n$.  Then the isomorphism type of\/ $\langle f_1,\ldots,f_n\rangle$ depends only on the left-to-right order of the intervals\/ $\{L_i,R_i\}_{i=1}^n$.
\end{theorem}

For example, suppose $\{f_1,f_2\}$ is a fast set with pairwise disjoint feet $L_1,R_1,L_2,R_2$.  Then up to symmetry there are exactly three possibilities:
\begin{enumerate}
    \item If $L_1<R_1<L_2<R_2$, then $f_1$ and $f_2$ have disjoint supports and hence commute, so $\langle f_1,f_2\rangle$ is isomorphic to $\mathbb{Z}\times \mathbb{Z}$.\smallskip
    \item If $L_1 < L_2 < R_2 < R_1$, then the orbital for $f_2$ lies in a fundamental domain for the action of $f_1$, so $\langle f_1,f_2\rangle$ is isomorphic to the wreath product $\mathbb{Z}\wr\mathbb{Z}$.\smallskip
    \item If $L_1<L_2<R_1<R_2$, then $\langle f_1,f_2\rangle$ is isomorphic to Thompson's group~$F$.
\end{enumerate}

\begin{figure}
    \renewcommand\thesubfigure{\arabic{subfigure}}
    \begin{center}
        \setlength{\unitlength}{3mm}
        \begin{subfigure}{0.3\textwidth}
            \centering
            \begin{picture}(10,5)
                \multiput(0,0)(4,0){2}{\circle*{0.5}}
                \multiput(5,0)(4,0){2}{\circle*{0.5}}
                \qbezier(0,0)(2,3)(4,0)
                \qbezier(5,0)(7,3)(9,0)
            \end{picture}
            \caption{$\mathbb{Z}\times\mathbb{Z}$}
        \end{subfigure}
        \hfill
        \begin{subfigure}{0.3\textwidth}
            \centering
            \begin{picture}(10,5)         
                \multiput(0,0)(8,0){2}{\circle*{0.5}}
                \multiput(2,0)(4,0){2}{\circle*{0.5}}
                \qbezier(0,0)(4,6)(8,0)
                \qbezier(2,0)(4,3)(6,0)
            \end{picture}
            \caption{$\mathbb{Z}\wr\mathbb{Z}$}
        \end{subfigure}
        \hfill
        \begin{subfigure}{0.3\textwidth}
            \centering
            \begin{picture}(10,5)
                \multiput(0,0)(6,0){2}{\circle*{0.5}}
                \multiput(2,0)(6,0){2}{\circle*{0.5}}
                \qbezier(0,0)(3,4.5)(6,0)
                \qbezier(2,0)(5,4.5)(8,0)
            \end{picture}
            \caption{$F$}
        \end{subfigure}
    \end{center}
    \caption{All possible support arrangements for fast sets of two bumps up to equivalence.}
    \label{fig:2bumps}
\end{figure}
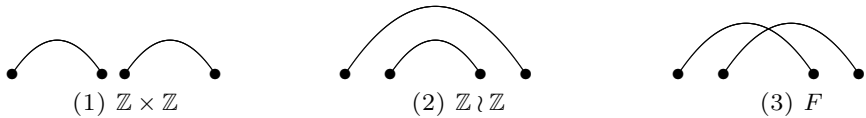

These three possibilities are shown in Figure~\ref{fig:2bumps}. The vertices represent the feet of the fast set while the arcs represent the bumps, each of which are connected to the vertices corresponding to the pair of feet for that particular bump. Importantly, the vertices are linearly ordered according to the left-to-right order on the corresponding feet.

In this paper, we connect the class of fast groups to the diagram groups introduced by Guba and Sapir~\cite{diagrams}.  There is one diagram group $\Delta(\mathcal{P},w)$ for each semigroup presentation $\mathcal{P}$ and each word $w$ in the generators of $\mathcal{P}$, with elements of the group being Van Kampen diagrams over $\mathcal{P}$ for the relation $w=w$.  The class of diagram groups also contains Thompson's group~$F$, with $F\cong \Delta(\mathcal{P},x)$ for $\mathcal{P}$ the semigroup presentation $\langle x\mid x^2=x\rangle$.  Our first main result is the following.

\begin{theorem}\label{thm:diagram}
Let\/ $\{f_1,\ldots,f_n\}$ be a fast set of positive bumps.  Then there exists a finite semigroup presentation $\mathcal{P}$ and word $w$ such that the fast group\/ $\langle f_1,\ldots,f_n\rangle$ is isomorphic to the diagram group\/ $\Delta(\mathcal{P},w)$.
\end{theorem}

Indeed, we give a simple algorithm to construct $\mathcal{P}$ and $w$ from the left-right order of the feet of $f_1,\ldots,f_n$.

As an application, we solve a question of Brin on the classification of fast groups~\cite[Question~108]{ober}.  In general, the class of groups generated by $n$ fast bumps includes many wreath products and direct products, as well a collection of groups $\mathfrak{C}_n$ that are neither of those.  For $n=2$, every group in the class $\mathfrak{C}_2$ is isomorphic to Thompson's group~$F$.  For $n=3$, Brin, Bleak, and Moore showed that every group in $\mathfrak{C}_3$ is isomorphic to the $3$-ary Thompson group~$F_3$ \cite[p.~1610]{ober}. For $n=4$ they showed that the groups in $\mathfrak{C}_4$ have at most two isomorphism classes, one of which is the $4$-ary Thompson group~$F_4$, and the other of which they named ``pseudo-$F_4$'' (see Figure~\ref{fig:FourBumps}).  However, they were not able to determine whether pseduo-$F_4$ is isomorphic to~$F_4$.
\begin{figure}[b]
    \centering
    \includegraphics{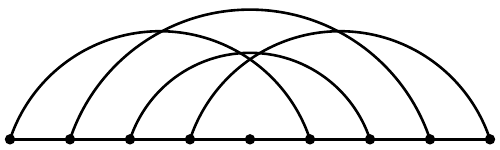}
    \caption{A fast set of four bumps that generate the group \mbox{pseudo-$F_4$}}
    \label{fig:FourBumps}
\end{figure}

Here we use the theory of diagram groups to prove the following.

\begin{theorem}
The group psuedo-$F_4$ is isomorphic to $F_4$.
\end{theorem}

Thus all of the groups in $\mathfrak{C}_4$ are isomorphic.  The proof involves explicit manipulation of semigroup presentations and is specific to the two groups in question, though similar calculations might be useful for larger values of~$n$. For $n\geq 5$, it remains an open question whether all of the groups in $\nbumps$ are isomorphic~\cite[Question~72]{ober}.  Even in the $n=5$ case, there at most four isomorphism classes in $\mathfrak{C}_5$, and it is not clear how to use the methods here to determine whether or not these classes are distinct. 

This paper is organized as follows.  In Section~\ref{sec:Preliminaries} we further describe geometrically fast groups, and briefly recall the definition of a diagram group.  In Section~\ref{sec:IsomorphismBumpsDiagrams}, we prove that every group of fast bumps is isomorphic to a diagram group of a certain type.  Finally, in Section~\ref{sec:IsomorphismF4PF4}, we use the theory of diagram groups to prove that pseudo-$F_4$ is isomorphic to~$F_4$.

\subsection*{Acknowledgements}
The authors would like to thank Collin Bleak, Matthew Brin, Justin Moore, and James Hyde for many helpful conversations and suggestions.

\section{Preliminaries}\label{sec:Preliminaries}

\subsection{Geometrically fast groups}

Here we recall the definition of fast groups given in~\cite{fast} in slightly more detail. See the introduction for definitions of a bump and an orbital, among others.

Let $\Homeo_+(I)$ be the group of all orientation-preserving homeomorphisms of the unit interval $I=[0,1]$. If $f\in\Homeo_+(I)$, the \textbf{support} of $f$ is the set $\supt(f)=\{x\in I\ |\ xf\neq x\}$. Given a collection of bumps $B=(b_i)\subset\Homeo_+(I)$ we may choose a collection of points $(x_i)$ such that $x_i\in \supt(b_i)$. We will refer to such a collection as a \textbf{marking} of $B$ and each element of the collection as a \textbf{marker}. Given a marker $x_i$ of a positive bump $b_i$ we define the \textbf{source} and \textbf{destination} of $b_i$ to be the intervals $\src(b_i)=(a,x_i)$ and $\dest(b_i)=[x_ib_i,c)$ respectively, where $\supt(b_i)=(a,c)$, and collectively refer to these intervals as the \textbf{feet of the bump $b_i$}. If $b_i$ is \textbf{negative} (i.e. not positive) then we define the source and destination of $b_i$ to be the destination and source of $b_i^{-1}$ respectively. We say a bump $b_i$ is \textbf{isolated} if its support doesn't contain any endpoints of bumps in $B$.

Finally, we say $B$ is a \textbf{geometrically fast} (or \textbf{fast} for short) set if there exists a marking for $B$ such that the feet defined by this marking are pairwise disjoint. Further, given an arbitrary set $A\subseteq \Homeo_+(I)$, if the set $B=\{a|_{O}\ |\ a\in A \text{, } O \text{ is an orbital of } a\}$ is geometrically fast then we say $A$ is geometrically fast as well, and in such a case the feet of $A$ are the feet of $B$. If a group of homeomorphisms $G$ is generated by a geometrically fast set then we also say the group $G$ is geometrically fast, or simply fast, and if all the functions in the set are bumps we may call $G$ a fast bump group, though for our purposes we will usually simply refer to them as fast groups. When considering a generating set for a geometrically fast bump group we will always assume the bumps are positive since this does not affect the isomorphism type of the group.

Given an ordered set $\mathcal{T}=(b_i)_i\subseteq B$ of bumps we say $\mathcal{T}$ is a \textbf{transition chain} if $a_i<a_{i+1}<c_i<c_{i+1}$ for all $i$ where $\supt(b_i)=(a_i,c_i)$. If, in addition, each interval $(a_{i+1},c_i)$ does not contain any endpoints of $B$ then we say $\mathcal{T}$ is \textbf{stretched}. Notice that the maximal stretched transition chains partition $B$. If $B$ is geometrically fast then we denote by $S(\mathcal{T})$ and $D(\mathcal{T})$ the sets of sources and destinations of the transition chain $\mathcal{T}$ respectively. These sets may inherit the ascending order from $\mathcal{T}$, which we denote by $\overrightarrow{S(\mathcal{T})}$ and $\overrightarrow{D(\mathcal{T})}$, or its inverse, denoted $\overleftarrow{S(\mathcal{T})}$ and $\overleftarrow{D(\mathcal{T})}$. We denote the minimal element of $\overrightarrow{S(\mathcal{T})}$ by $s(\mathcal{T})$ and the maximal element of $\overrightarrow{D(\mathcal{T})}$ by $d(\mathcal{T})$, and we may say the transition chain $\mathcal{T}$ \textbf{starts at} $s(\mathcal{T})$ and \textbf{ends at} $d(\mathcal{T})$. It will be convenient to consider the empty transition chain $\mathcal{T}=\emptyset$ as starting at every source of $B$ and ending at every destination of $B$.

Given a geometrically fast set $A$ we say that a word $w\in (A^\pm)^*$ is \textbf{simply locally reduced at a point $x\in I$} if for any prefix of the form $ua$ for $a\in A^\pm$ we have $xu\neq xua$. Relatedly, given a word $w\in (A^\pm)^*$ we define the \textbf{simply local reduction} $w_x$ of $w$ at $x\in I$ to be the word obtained by deleting $a$ from $w$ whenever $xu=xua$ for some prefix $ua$ of $w$. We denote the set of simply local reductions of a word by $L(w)=\{w_x\ |\ x\in I\}$. For a word $w$ we denote its free reduction by $w^\vee$; additionally, if a word is simply locally reduced at $x$ and freely reduced we say it is \textbf{locally reduced at $x$} and we say $(w_x)^\vee$ is a \textbf{local reduction} of $w$ while denoting its \textbf{set of local reductions} by $L^\vee(w)$.

We will require the notion of a history of a point. Given a fast set $A$ with a marking we say a point $x\in I$ outside of the feet of $A$ has \textbf{trivial history} and define $\bar{x}=\{a\ |\ x\in \supt(a)\}$. If $x\in I$ is a point in a foot of $A$ we define its \textbf{history} to be the set $\eta(x)$ containing
    
\begin{itemize}
    \item $u\in (A^\pm)^*$ where there is a $t\in I$ such that $tu=x$, $u$ is locally reduced at $t$ and $t$ is not in the source of the first letter of $u$
    \item $au\in (A^\pm)^*$ where $a\in\bar{t}$ for a $t\in I$ with trivial history, $tu=x$ and $u$ is locally reduced at $t$.
\end{itemize}

Given a fast set $A\subseteq \Homeo_+(I)$ one can define a directed, edge-labelled, vertex-ordered graph which represents what will turn out to be the essential features of $A$ for our purposes. This graph has a vertex corresponding to each foot of $A$ and for each $a\in A$ there is an edge (with label $a$) from the vertex corresponding to the source of $a$ to the vertex corresponding to the destination of $a$. Crucially, the vertices are ordered isomorphically to the ordering on the feet induced from the usual order on $I$. We refer to this as the \textbf{dynamical diagram} $D_A$ of $A$ (unrelated to the diagrams which make up diagram groups). Examples of dynamical diagrams are shown in Figure~\ref{fig:2bumps} and Figure~\ref{fig:FourBumps}.

\subsection{Diagram groups}

\subsubsection{Diagrams}

Diagrams and diagram groups were originally explored in the PhD thesis of Kilibarda~\cite{kili} and the theory around these groups has been developed chiefly by the work of Guba and Sapir (e.g.~\cite{diagrams}).  See also Genevois's recent survey on diagram groups~\cite{Gen}.

Diagrams are perhaps best understood as two-dimensional analogues of words, and diagram groups, similarly, as analogous to free groups. To define a diagram, we may start with a semigroup presentation 
\[
\langle x_1,\ldots,x_m \mid u_1=v_1,\;\ldots,\;u_n=v_n\rangle
\]
and consider a word $w=x_{i_1}\ldots x_{i_k}$ over this presentation. This word can be represented as an edge-labelled oriented plane graph by a positive path $\Delta_0$ with $k$ edges with the $j$th edge from left to right labelled $x_{i_j}$. If $w$ contains a subword $u_i$ then we can replace it with $v_i$ to obtain a word $w'$ equivalent to $w$. This can be represented in the graph by connecting a positive path labelled $v_i$ beneath the subpath labelled $u_i$ to obtain a new graph $\Delta_1$. This can be repeated for $\Delta_1$, and so on, indefinitely---each graph $\Delta_0,\Delta_1,\ldots$ obtained is a diagram over the semigroup presentation.

As an example, consider the semigroup presentation $\langle a,b\ |\ ab=ba\rangle$ and the derivation \[aabb\rightarrow abab\rightarrow baab\rightarrow baba\rightarrow bbaa\] from the word $aabb$ to $bbaa$. The diagram defined by this derivation over this presentation is shown in Figure~\ref{f:diagram}.

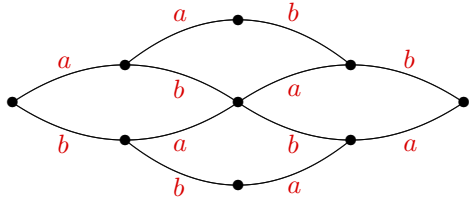
\begin{figure}
    \centering
        \setlength{\unitlength}{2mm}
        \begin{picture}(30,15)
            \multiput(7.5,10)(15,0){2}{\circle*{0.7}}
            \qbezier(7.5,10)(15,16)(22.5,10)
            \put(15,13){\circle*{0.7}}
            
            \multiput(0,7.5)(15,0){3}{\circle*{0.7}}
            \qbezier(0,7.5)(7.5,12.5)(15,7.5)
            \qbezier(15,7.5)(22.5,12.5)(30,7.5)
            
            \qbezier(0,7.5)(7.5,2.5)(15,7.5)
            \qbezier(15,7.5)(22.5,2.5)(30,7.5)
            
            \multiput(7.5,5)(15,0){2}{\circle*{0.7}}
            \qbezier(7.5,5)(15,-1)(22.5,5)
            \put(15,2){\circle*{0.7}}
            
            \color{myred}
            \put(3,9.8){$a$}
            \put(10.7,12.9){$a$}
            \put(18.3,12.9){$b$}
            \put(26,9.8){$b$}
            
            \put(10.7,7.8){$b$}
            \put(18.3,7.9){$a$}
            
            \put(3,4.15){$b$}
            \put(10.7,4.3){$a$}
            \put(18.3,4.15){$b$}
            \put(26,4.3){$a$}
            
            \put(10.7,1.3){$b$}
            \put(18.3,1.45){$a$}
        \end{picture}
    \caption{\label{f:diagram} An $(aabb, bbaa)$-diagram over the semigroup presentation $\langle a,b \mid ab=ba\rangle$}
\end{figure}

The bounded faces of a diagram are called \textbf{cells} and are bounded by the disjoint union of two positive paths---the \textbf{top path} and the \textbf{bottom path} of the cell. A cell $\pi$ corresponds to a relation of the presentation---its top path $\lceil\pi\rceil$ is labelled by one side while its bottom path $\lfloor\pi\rfloor$ is labelled by the other. Cells are also labelled with the relations they correspond to, though this isn't routinely included in pictures since it is implicit from the path labels. We sometimes use $\pi_{u,v}$ to denote a cell with top path label $u$ and bottom path label $v$. Similar to a cell, a diagram $\Delta$ has designated top $\lceil\Delta\rceil$ and bottom $\lfloor\Delta\rfloor$ paths whose disjoint union bounds the entire graph; if a diagram has its top path with label $u$ and its bottom path with label $v$ then we call it a $(u,v)$-diagram.

If we have a $(u,v)$-diagram $\Delta_1$ and a $(v,w)$-diagram $\Delta_2$ we define their \textbf{concatenation} $\Delta_1\circ\Delta_2$ to be the $(u,w)$-diagram obtained by identifying the bottom path of $\Delta_1$ with the top path of $\Delta_2$. If we have a $(u_1,v_1)$-diagram $\Delta_1$ and a $(u_2,v_2)$-diagram $\Delta_2$ then we define their \textbf{addition} $\Delta_1+\Delta_2$ to be the $(u_1u_2,v_1v_2)$-diagram obtained by identifying the rightmost vertex (the terminal vertex $t(\Delta_1)$) of $\Delta_1$ with the leftmost vertex (the initial vertex $i(\Delta_2)$) of $\Delta_2$.

If a diagram has two cells $\pi_1,\pi_2$ such that the bottom path of $\pi_1$ is the top path of $\pi_2$ and their opposing paths have the same label we say they form a \textbf{dipole}. We can remove a dipole by deleting the path which the cells share and identifying their remaining paths. We consider diagrams equivalent up to removing dipoles and we call a diagram \textbf{reduced} if it contains no dipoles. It was shown in \cite{kili} that reduced diagrams provide a normal form for these equivalence classes.

We denote by $D(\mathcal{P},w)$ the set of all equivalence classes of $(w,w)$-diagrams over the presentation $\mathcal{P}$. This forms a group under concatenation where the identity is the positive path labelled $w$, known as the \textbf{trivial diagram} $\varepsilon_w$, and the inverse of a diagram $\Delta$ is obtained by reflecting it along a horizontal axis.

We will require some definitions regarding the vertices of a diagram. Given a vertex $v$ we may consider its incoming edges and its outgoing edges, the sets of which are denoted $I(v)$ and $O(v)$ respectively. Notice that the cyclic order of the edges around $v$ separates each into a connected set, as such each has a natural linear order: $I(v)$ is endowed with the counter-clockwise order around $v$, while $O(v)$ is endowed with the clockwise, both starting at the point in the cyclic clockwise order where $I(v)$ gives way to $O(v)$. By the same token we can consider the set $I^\Pi(v)$ of cells who have $v$ as their terminal vertex and the set $O^\Pi(v)$ of cells who have $v$ as their initial vertex, each of which inherit the ordering from the edges of $v$ which separate consecutive cells. 

\subsubsection{Strand diagrams}

For our purposes it will be useful to consider an alternative pictorial representation of elements of diagram groups.  These are known as \textbf{strand diagrams}, a term coined by the first author in~\cite{Belk} (see also~\cite{strands} for a more general version).  Strand diagrams are closely related to the ``transistor'' pictures for diagram groups described by Guba and Sapir in \mbox{\cite[Section~4]{diagrams}} and equivalent to the planar subgroup of the braided diagram groups they later define in \mbox{\cite[Section~16]{diagrams}} (as a remark, it has been suggested in \cite{Gen} that symmetric diagram groups would be a more appropriate name for these objects). Given a diagram $\Delta$ over a semigroup presentation $\mathcal{P}$ we define its \textbf{strand diagram} $\Psi$ as follows

\begin{enumerate}
    \item There is a vertex for every cell of $\Delta$;\smallskip
    \item There is an edge from $v_1$ to $v_2$ if the corresponding cells $\pi_1,\pi_2$ have a shared boundary edge $e\in\lfloor\pi_1\rfloor\cap\lceil\pi_2\rceil$ in $\Delta$;\smallskip
    \item There is a vertex for each cell on the top (bottom) path, and an edge from (to) such a vertex to (from) another if the corresponding top (bottom) path edge forms part of the boundary of the corresponding cell.
\end{enumerate}

\noindent
and we label each component of $\Psi$ with the label of the corresponding component of~$\Delta$. We refer to the vertices defined in (1) as \textbf{interior vertices} and those defined in (3) as \textbf{boundary vertices}. An example of a strand diagram is shown in Figure~\ref{f:strand}. Statements (1) and (2) define the dual of $\Delta$ while statement (3) provides a tweak which allows us to concatenate strand diagram as we would diagrams. We can define dipoles and reductions in much the same way as with diagrams, and as such we can define the strand diagram group $S(\mathcal{P},w)$ as the group of equivalence classes of $(w,w)$-strand diagram over $\mathcal{P}$. It is clear that $S(\mathcal{P},w)\cong D(\mathcal{P},w)$.

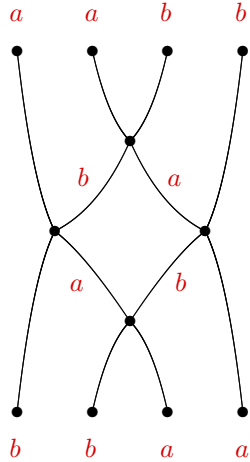
\begin{figure}
    \centering
        \setlength{\unitlength}{2mm}
        \begin{picture}(15,30)
            \multiput(0,28)(5,0){4}{\circle*{0.7}}
            \put(7.5,22){\circle*{0.7}}
            \multiput(2.5,16)(10,0){2}{\circle*{0.7}}
            \put(7.5,10){\circle*{0.7}}
            \multiput(0,4)(5,0){4}{\circle*{0.7}}
            
            \qbezier(0,28)(1,19)(2.5,16)
            \qbezier(15,28)(14,19)(12.5,16)
            \qbezier(2.5,16)(1,13)(0,4)
            \qbezier(12.5,16)(14,13)(15,4)
            
            \qbezier(5,28)(6,23.5)(7.5,22)
            \qbezier(10,28)(9,23.5)(7.5,22)
            \qbezier(5,4)(6,8.5)(7.5,10)
            \qbezier(10,4)(9,8.5)(7.5,10)
            
            \qbezier(7.5,22)(5.5,17.5)(2.5,16)
            \qbezier(7.5,22)(9.5,17.5)(12.5,16)
            \qbezier(2.5,16)(4.5,14.5)(7.5,10)
            \qbezier(12.5,16)(10.5,14.5)(7.5,10)
            
            \color{myred}
            \put(-0.5,30){$a$}
            \put(4.5,30){$a$}
            \put(9.5,30){$b$}
            \put(14.5,30){$b$}
            
            \put(-0.5,1){$b$}
            \put(4.5,1){$b$}
            \put(9.5,1){$a$}
            \put(14.5,1){$a$}
            
            \put(4,19){$b$}
            \put(10,19){$a$}
            \put(3.5,12){$a$}
            \put(10.5,12){$b$}
        \end{picture}
    \caption{The strand diagram corresponding to the diagram shown in Figure~\ref{f:diagram}}
    \label{f:strand}
\end{figure}

We borrow an idea from \cite{strands} characterising interior vertices. If an interior vertex has precisely one incoming edge and at least two outgoing edges we call it a \textbf{split}. Similarly, if an interior vertex has precisely one outgoing edge and at least two incoming edges we call it a \textbf{merge}. Notice that if we have a presentation $\mathcal{P}=\langle\Sigma\ |\ \mathcal{R}\rangle$ such that $\mathcal{R}$ only has at most one relation of the form $u=w$ for each $u\in\Sigma$ where $w$ has length at least two (such a $\mathcal{P}$ is known as \textbf{tree-like}) then the above distinction dichotomises the interior vertices of any strand diagram over $\mathcal{P}$. Indeed, the presentations considered in this paper have this property.

\subsection{Some isomorphism theorems for diagram groups}

In \cite[Theorem~7.1]{diagrams}, Guba and Sapir prove that if $w$ and $w'$ are words over the generators of $\mathcal{P}$ that represent the same element of the corresponding semigroup, then the diagram groups $D(\mathcal{P},w)$ and $D(\mathcal{P},w')$ are isomorphic.  Moreover, in \mbox{\cite[Theorem~4.1]{complex}}, Guba and Sapir describe a very general method of modifying a semigroup presentation to yield an isomorphic diagram group.  Among the consequences of this theorem are the following.

\begin{theorem}[Guba and Sapir Moves \cite{complex}]\label{thm:moves}
    Consider a semigroup presentation $\mathcal{P}$.
    \begin{enumerate}
        \item Suppose $u_i=v_i$ and $u_j=v_j$ are distinct relations in $\mathcal{P}$.  If either $u_j$ or $v_j$ has $u_i$ as a subword, then replacing this $u_i$ subword with $v_i$ in the relation $u_j=v_j$ does not change the isomorphism type of $D(\mathcal{P},u)$ for any word $u$.\smallskip
        \item Suppose $x$ is a generator in $\mathcal{P}$ and only appears in one relation, which has the form $x=w$ for some word $w$ that does not contain~$x$.  In this case, removing the generator $x$ as well as the relation $x=w$ does not change the isomorphism type of $D(\mathcal{P},u)$ for any word $u$ not containing $x$.
    \end{enumerate}
\end{theorem}

Note that both of these moves resemble Tietze transformations of group presentations, in that they don't change the isomorphism type of the underlying semigroup.  However, the reader should be warned that transformations that preserve the isomorphism type of the semigroup can nonetheless change the isomorphism type of an associated diagram group.  For example, adding a new relation to a semigroup presentation that follows from the existing relations will usually change the diagram group.

\section{Geometrically fast bump groups are isomorphic to diagram groups}\label{sec:IsomorphismBumpsDiagrams}

\subsection{The diagram representation}

Consider a geometrically fast set of bumps $B$ with dynamical diagram $D$. Then there exists a marking of $B$ which witnesses its sources and destinations being disjoint. Taking the sources and destinations along with the connected subsets of their complement in the support of $B$ (i.e. the gaps between the feet) we obtain a general partition of the support of $B$.

We can improve on this partition using Theorem~\ref{thm:PingPong}. Since any geometrically fast set with the same dynamical diagram generates the same group, a partition obtained as above from a particular set can be used as a general partition associated to the dynamical diagram and the group it generates. In particular, for any dynamical diagram without isolated bumps there exists a geometrically fast set with that diagram such that there exists a marking which defines disjoint feet that cover the support of the set (modulo finitely many isolated points). An example is shown in Figure~\ref{f:cp}. For dynamical diagrams with isolated bumps this is nearly possible---in this case, we just require a `gap' (i.e. a fundamental domain) between the source and destination of each isolated bump. We may call this the \textbf{canonical partition} of a dynamical diagram and the marking which realises it may be called its \textbf{canonical marking} (this may be taken to coincide with the canonical marking as defined in \cite{fast}).

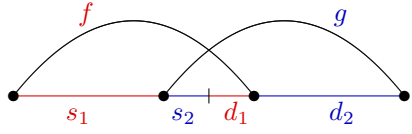
\begin{figure}
    \centering
        \setlength{\unitlength}{2mm}
        \begin{picture}(30,7)
            \color{myred}
            \put(4.3,5){$f$}
            \put(3.5,-1.6){$s_1$}
            \put(14,-1.6){$d_1$}
            \put(0,0){\line(1,0){10}}
            \put(16,0){\line(-1,0){3}}
            
            \color{myblue}
                
            \put(21.3,5){$g$}  
            \put(10.5,-1.6){$s_2$}
            \put(21,-1.6){$d_2$}

            \put(10,0){\line(1,0){3}}
            \put(26,0){\line(-1,0){10}}

            \color{black}
            \put(13,-0.5){\line(0,1){1}}
            
            \multiput(0,0)(16,0){2}{\circle*{0.7}}
                \qbezier(0,0)(8,10)(16,0)
            \multiput(10,0)(16,0){2}{\circle*{0.7}}
            \qbezier(10,0)(18,10)(26,0)
        \end{picture}
    \caption{The \rule{0pt}{14pt}canonical partition for the dynamical diagram generating~$F$.}
    \label{f:cp}
\end{figure}

Having obtained such a partition for a dynamical diagram we may use it to define diagrams (in the sense of Guba and Sapir) representing each bump in $B$. This is best described through a simple example. Consider the partition of the dynamical diagram shown in Figure~\ref{f:cp}. By the definition of this partition we have \[(s_1)f=s_1s_2,\text{ } (s_2d_1)f=d_1\] and \[(s_2)g=s_2d_1,\text{ } (d_1d_2)g=d_2\] where in each case restrictions of the maps witness the intervals as homeomorphic. This observation makes it natural to consider the diagrams

\[\beta_f=\pi_{A,AB}+\pi_{BC,C}+\varepsilon_{D}\]
\[\beta_g=\varepsilon_{A}+\pi_{B,BC}+\pi_{CD,D}\]

\medskip

\noindent
over the semigroup presentation $\mathcal{P}=\langle A,B,C,D\ |\ A=AB, B=BC, C=BC, D=CD\rangle$ as candidates for representing the bumps $f$ and $g$. Pictures of these diagrams are shown in Figure~\ref{f:gd}.

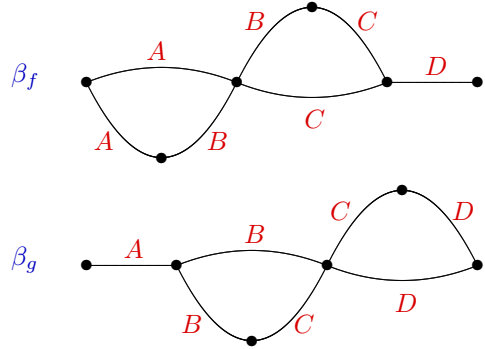
\begin{figure}
    \begin{center}
    \setlength{\unitlength}{2mm}
        \begin{picture}(30,16)
            \multiput(0,10)(10,0){3}{\circle*{0.7}}
            \put(26,10){\circle*{0.7}}
            \put(5,5){\circle*{0.7}}
            \put(15,15){\circle*{0.7}}
            \qbezier(0,10)(5,12)(10,10)
            \qbezier(0,10)(5,0)(10,10)
            \qbezier(10,10)(15,20)(20,10)
            \qbezier(10,10)(15,8)(20,10)
            \qbezier(20,10)(23,10)(26,10)
            \color{myred}
            \put(4,11.5){$A$}
            \put(10.5,13.5){$B$}
            \put(18,13.5){$C$}
            \put(22.5,10.5){$D$}
            \put(0.5,5.5){$A$}
            \put(8,5.5){$B$}
            \put(14.5,7){$C$}
            \color{myblue}
            \put(-5,10){$\beta_f$}
        \end{picture}
        
        \begin{picture}(30,7)
            \multiput(6,5)(10,0){3}{\circle*{0.7}}
            \put(0,5){\circle*{0.7}}
            \put(11,0){\circle*{0.7}}
            \put(21,10){\circle*{0.7}}
            \qbezier(0,5)(3,5)(6,5)
            \qbezier(6,5)(11,-5)(16,5)
            \qbezier(6,5)(11,7)(16,5)
            \qbezier(16,5)(21,15)(26,5)
            \qbezier(16,5)(21,3)(26,5)
            \color{myred}
            \put(2.5,5.5){$A$}
            \put(10.5,6.5){$B$}
            \put(16.2,8){$C$}
            \put(24.3,8){$D$}
            \put(6.3,0.5){$B$}
            \put(13.8,0.5){$C$}
            \put(20.5,1.9){$D$}
            \color{myblue}
            \put(-5,5){$\beta_g$}
        \end{picture}
    \end{center}
    \caption{\label{f:gd} The diagrams representing the bumps $f,g$ from Figure~\ref{f:cp}}
\end{figure}

We now describe the semigroup presentation and diagrams defined by a dynamical diagram of a geometrically fast set of bumps $B$ in full generality. Suppose $B$ contains $n$ bumps, $k$ of which are isolated, and let $A_1,\ldots,A_{2n+k}$ be the canonical partition of its dynamical diagram. Consider $b\in B$ and define $i(b)$ and $j(b)$ to be the integers such that $\src(b)=A_{i(b)}$ and $\dest(b)=A_{j(b)}$. Then $A_{i(b)},A_{i(b)+1},\ldots,A_{j(b)-1},A_{j(b)}$ is the partition of $\supt(b)$ contained in the canonical partition and we define $G(b)=A_{i(b)+1}\ldots A_{j(b)-1}$, which is a fundamental domain of $b$. By definition \[(A_i(b))b=A_i(b)G(b), (G(b)A_{j(b)})b=A_{j(b)}\] and so the diagram representing $b$ is \[\beta_b=\varepsilon_{A_1\ldots A_{i(b)-1}}+\pi_{A_i(b),A_i(b)G(b)}+\pi_{G(b)A_{j(b)},A_{j(b)}}+\varepsilon_{A_{j(b)+1}\ldots A_{2n+k}}\] over the semigroup presentation $\mathcal{P}=\langle A_1,\ldots, A_{2n+k}\ |\ A_{i(b)}=A_{i(b)}G(b), G(b)A_{j(b)}=A_{j(b)} \text{ for each } b\in B\rangle$. We may refer to the $\beta_b$ as \textbf{generator diagrams}.

We claim that $\{\beta_b\ |\ b\in B\}$ is a generating set for the diagram group $D(\mathcal{P},A_1\ldots A_{2n+k})$ which is isomorphic to the geometrically fast group $\langle B\rangle$.

\subsection{The isomorphism}

Let $B$ be a finite geometrically fast set of bumps and let $a_1,\ldots,a_n$ and $\mathcal{P}$ be the canonical partition and semigroup presentation obtained from $B$ as outlined in the previous section. We let $\delta: B\rightarrow D(\mathcal{P},a_1\ldots a_n)$ be the map defined by $b_i\mapsto\beta_i$ where $\beta_i$ is the diagram representing the bump $b_i$. We prove that this map extends to an isomorphism between $\langle B\rangle$ and $D(\mathcal{P},a_1\ldots a_n)$.

For the purposes of the following proposition we add an additional set of labels to our strand diagrams. Notice from how $\mathcal{P}$ is constructed there is a two-to-one correspondence between its relations and the set of bumps $B$. Thus, we will label each interior vertex of a strand diagram $S$ with the bump corresponding to the relation it is labelled by (recalling that cells of diagrams, and therefore interior vertices of strand diagrams, are labelled by relations). We denote by $BP(S)$ the set of bump vertex labels of maximal paths in $S$; that is, the paths starting at a top boundary vertex and ending at a bottom boundary vertex.

\begin{prop}
    Let $w\in B^+$ where $B$ is a fast set of bumps and let $S$ be the reduced strand diagram of the image $w\delta$. There is an interval partition $\mathscr{P}_1,\ldots,\mathscr{P}_k$ of the support of $B$ in left-to-right, one-to-one correspondence with the set of maximal paths $p_1,\ldots,p_k$ of $S$ such that given any $x\in\mathscr{P}_i$ the simply local reduction $w_x$ is the bump vertex label of $p_i$. Furthermore, if $p_i$ is a path with boundary labels $a_1$ and $a_2$ respectively then $\mathscr{P}_i\subseteq a_1$ and $(\mathscr{P}_i)w\subseteq a_2$.
\end{prop}

\begin{proof}
    We proceed by induction on the length of $w$. Let $b\in B$ be supported on $A=a_i\ldots a_j$ and set $\beta=b\delta$ and, without loss, suppose $S\circ\beta$ does not contain dipoles. If a maximal path on $S$ ends at a boundary vertex labelled $a\notin\{a_i,\ldots,a_j\}$ then it is preserved in $S\circ\beta$ and, similarly, if we take a point $x$ such that $xw\notin A$ then $(wb)_x=w_x$.
    
    Suppose $p$ is a maximal path on $S$ ending at a vertex labelled by $a\in\{a_i,\ldots,a_j\}$ and let $\mathscr{P}$ be its corresponding part. In $S\circ\beta$, consider those maximal paths which contain $p$. If there is only one, call it $q$, its final interior vertex must be a merge and we take $\mathscr{P}$ as corresponding to $q$. Since the bump vertex label of $p$ is equal to the simply local reduction $w_x$ of $w$ for any $x\in\mathscr{P}$ and $(\mathscr{P})w\subseteq a$ we see that $(wb)_x=w_xb$ is the bump vertex label of $q$ for any $x\in\mathscr{P}$. Further, the final vertex of $q$ being a merge corresponds to the action of $b$ sending the interval $a$ into the destination of $b$, which labels the bottom boundary vertex of $q$.
    
    If $p$ is contained in $d$ distinct maximal paths $q_1,\ldots,q_d$ of $S\circ\beta$ then they must share a final interior vertex and it must be a split. First notice that each $q_l$ must only contain splits since otherwise somewhere a merge would immediately precede a split but it follows from $\mathcal{P}$ being tree-like that it would form a dipole in $S\circ\beta$.
    
    \medskip
    
    \leftskip 10pt
    \rightskip 10pt
    
    \noindent\textit{Claim}: If $p$ is a path in $S$ from $a_1$ to $a_2$ consisting only of splits then its corresponding part $\mathscr{P}$ satisfies $\mathscr{P}w=a_2$.
    
    \smallskip
    
    \noindent\textit{Proof}: By the inductive hypothesis we have that $\mathscr{P}w\subseteq a_2$. If $\mathscr{P}w\subset a_2$ then there would have to be some other part $\mathscr{P}'$ which is also mapped into $a_2$ by $w$---however, since $p$ consists only of splits, any other path must lead to a boundary vertex labelled by a distinct part of the canonical partition and so, by the hypothesis, would not map into $a_2$. Thus, $\mathscr{P}w=a_2$.
    
    \leftskip 0pt
    \rightskip 0pt
    
    \medskip
    \noindent
    The final vertex of the paths $q_1,\ldots,q_d$ being a split corresponds to the action of $b$ mapping the interval $a$, which is its source, onto the complement of its destination inside $A$. By the claim, this means that $\mathscr{P}w$ is mapped onto a subset of the canonical partition such that we may partition $\mathscr{P}$ into intervals $\mathscr{Q}_1,\ldots,\mathscr{Q}_d$ with $\mathscr{Q}_lw=a_l$ where $a_l$ is the label of the boundary vertex of $q_l$. We take $\mathscr{Q}_l$ to be the part corresponding to $q_l$ and note that if $x\in\mathscr{Q}_l$ we have that $w_x$ is the bump vertex label of $p$, from which it follows that $w_xb$ is the bump vertex label of $q_l$ and it is clear that $w_xb$ is simply locally reduced with respect to $x$.
\end{proof}

\noindent
One useful upshot of this proposition is the following corollary.

\begin{cor}
    Let $w\in B^+$ and let $S$ be a strand diagram of the image $w\delta$. Then \[L^\vee(w)=\{u^\vee\ |\ u\in BP(S)\}\]
\end{cor}

If we consider the details of last proof we can be more precise regarding the structure of the partition $\mathscr{P}_1,\ldots,\mathscr{P}_k$ for a given $w$. Essentially, the partition is constructed by inductively refining the canonical partition---following a path starting at the top boundary of its strand diagram, each time we cross a split labelled by $b\in B$ we partition the current part $\mathscr{P}$ in question according to which part of the canonical partition the action of $b$ sends each point to; that is, we break $\mathscr{P}$ into disjoint intervals where the breakpoints are the images of the breaks in the canonical partition under $b^{-1}$. Since the canonical partition is defined by markers and images of markers we can then see that the partition $\mathscr{P}_1,\ldots,\mathscr{P}_k$ is defined by a finite subset of the set of orbits of markers $M\langle B\rangle$ where $M$ is the canonical marking. We may refer to such a partition as the \textbf{canonical refinement} for the word $w$.

\begin{cor}
    Let $w\in B^*$ where $B$ is a geometrically fast set of bumps and let $M$ be a marking which witnesses that $B$ is fast. Then \[L(w)=\{w_x\ |\ x\in M\langle B\rangle \}\]
\end{cor}

\begin{proof}
    The set of simply local reductions is independent of choice of marking so we may choose $M$ to be canonical. Consider the canonical refinement $\mathscr{R}$ of $w$ and its reduced strand diagram $S$. The refinement is defined by some finite subset $R\subset M\langle B\rangle$ and by Proposition 3.1 there is a surjection $\mathscr{R}\rightarrow L(w)$ such that if $x\in\mathscr{P}\in\mathscr{R}$ then $\mathscr{P}\mapsto w_x$. If one of the parts $\mathscr{P}$ does not contain its boundary then we can see from how canonical refinements are constructed that we may further subdivide this part via a split so that the breakpoints, contained in $\mathscr{P}$, are elements of $M\langle B\rangle$.
\end{proof}

\begin{prop}
    The map $\delta: \langle B\rangle\rightarrow D(\mathcal{P},a_1\ldots a_n)$ is a homomorphism.
\end{prop}

\begin{proof}
    Let $w$ be a word on $B$ such that $w\equiv 1$ and consider the image $w\delta$. We want to show that the reduced strand diagram $S$ of $w\delta$ is trivial. By Corollary 3.2 it suffices to show that $L^\vee(w)=\{1\}$ and, by Corollary 3.3, it then suffices to show $w^\vee_x=1$ for $x\in M\langle B\rangle$.
    
    Consider $x\in M\langle B\rangle$. Since markers have trivial history $x$ must have finite history. By \cite[Lemma~5.7]{fast} if $y$ is a point in the orbit of $x$ then there must be precisely one word $u$ locally reduced at $x$ such that $xu=y$. By assumption we have $xw^\vee_x=xw=x$ and we therefore conclude that $w^\vee_x=1$.
\end{proof}

\noindent
The following lemma, which will be useful for the proof of $\delta$ being surjective, describes in detail the possible labellings of the edges and cells incident to a given vertex in a diagram over $\mathcal{P}$. For such a vertex $v$ notice that the order on the sets $I^\Pi(v)$ and $O^\Pi(v)$ must consist of a (possibly empty) sequence of positive cells followed by a (possibly empty) sequence of negative cells. As such there exists an edge $e\in I(v)$, and similarly for $O(v)$, which marks the point where the sequence $I^\Pi(v)$ changes from positive to negative, and we will refer to $e$ as the \textbf{inflection edge} of $I^\Pi(v)$.

\begin{lemma}\label{lemma:vertex}
    Let $\Delta$ be a reduced $(a_1\ldots a_n,a_1\ldots a_n)$-diagram over $\mathcal{P}$ and consider $v\neq i(\Delta),t(\Delta)$ one of its vertices of degrees at least three. Then there exists $k$ such that $a_{k-1}$ is the label of the first edge in $I(v)$ and $a_k$ is the label of the first edge in $O(v)$. Further,
    \begin{enumerate}
        \item if $a_k$ is neither a source nor destination then $a_{k-1}$ must be a source and the sequence of labels on $I(v)$ is $a_{k-1},a_k$ while the sequence of labels on $O(v)$ is $a_k,a_{k+1}$ and the singletons $I^\Pi(v),O^\Pi(v)$ are positive and negative respectively;
        \item if $a_{k-1}$ is neither a source nor destination then $a_k$ must be a destination and the sequence of labels on $O(v)$ is $a_k,a_{k-1}$ while the sequence of labels on $I(v)$ is $a_{k-1},a_{k-2}$ and the singletons $I^\Pi(v),O^\Pi(v)$ are negative and positive respectively;
        \item if $a_{k-1}$ is a destination and $a_k$ is a source then every $e\in I(v)$ is labelled $a_{k-1}$ while every $e\in O(v)$ is labelled $a_k$ and for both $I^\Pi(v)$ and $O^\Pi(v)$ the inflection edge occurs at either the first edge or the last;
        \item if $a_{k-1}$ is a source then let $\mathcal{T}_m$ be the maximal stretched transition such that $a_{k-1}\in S(\mathcal{T}_m)$. Then there exists two stretched transition chains $\mathcal{T}_1,\mathcal{T}_2$ such that $\mathcal{T}_1$ starts at $a_{k-1}$ and the sequence of labels on $I(v)$ in order is $\overrightarrow{S(\mathcal{T}_1)}$, followed by some number $m_1\geq 0$ of $a_{m-1}$ if $a_m=d(\mathcal{T}_1)=d(\mathcal{T}_m)$, where the first or the last may be the inflection edge, and then $\overleftarrow{S(\mathcal{T}_2)}$ where $\mathcal{T}_2$ ends at $d(\mathcal{T}_1)$, the first of which is the inflection edge if it has not yet occurred, while
        \begin{itemize}
            \item if $\mathcal{T}_2\subseteq\mathcal{T}_1$ the sequence of labels on $O(v)$ is some number $m_2\geq 1$ of $a_k$ (if $a_k$ is a destination then $m_2=1$), either the first or the last one being the inflection edge, followed by $\overrightarrow{D(\mathcal{T}_1\setminus\mathcal{T}_2)}$ but;
            \item if $\mathcal{T}_1\subseteq\mathcal{T}_2$ then $O(v)$ is some number $m_2\geq 1$ of $a_k$ (if $a_k$ is a destination then $m_2=1$) followed by $\overleftarrow{D(\mathcal{T}_2\setminus\mathcal{T}_1)}$ where the last may be the inflection edge and, finally, some number $m_3\geq 0$ of $a_{m+1}$ if $a_m=s(\mathcal{T}_2)=s(\mathcal{T}_m)$, the first or the last being the inflection edge if it has not yet occurred;
        \end{itemize}
        \item if $a_k$ is a destination then let $\mathcal{T}_m$ be the stretched transition chain such that $a_k\in D(\mathcal{T}_m)$. Then there exists two stretched transition chains $\mathcal{T}_1,\mathcal{T}_2$ such that $\mathcal{T}_1$ ends at $a_k$ and the sequence of labels on $O(v)$ in order is $\overleftarrow{D(\mathcal{T}_1)}$, followed by some number $m_1\geq 0$ of $a_{m+1}$ if $a_m=s(\mathcal{T}_1)=s(\mathcal{T}_m)$, where the first or the last may be the inflection edge, and then $\overrightarrow{D(\mathcal{T}_2)}$ where $\mathcal{T}_2$ begins at $s(\mathcal{T}_1)$, the first of which is the inflection edge if it has not yet occurred, while 
        \begin{itemize}
            \item if $\mathcal{T}_2\subseteq\mathcal{T}_1$ the sequence of labels on $I(v)$ is some number $m_2\geq 1$ of $a_{k-1}$ (if $a_{k-1}$ is a source then $m_2=1$), either the first or the last one being the inflection edge, followed by $\overleftarrow{S(\mathcal{T}_1\setminus\mathcal{T}_2)}$ but;
            \item if $\mathcal{T}_1\subseteq\mathcal{T}_2$ then $I(v)$ is some number $m_2\geq 1$ of $a_{k-1}$ (if $a_{k-1}$ is a source then $m_2=1$) followed by $\overrightarrow{S(\mathcal{T}_2\setminus\mathcal{T}_1)}$ where the last may be the inflection edge and, finally, some number $m_3\geq 0$ of $a_{m-1}$ if $a_m=d(\mathcal{T}_2)=d(\mathcal{T}_m)$, the first or the last being the inflection edge if it has not yet occurred.
        \end{itemize}
    \end{enumerate}
\end{lemma}

\begin{proof}
    Recalling that, since $\mathcal{P}$ is tree-like, $\Delta=\Delta_1\circ\Delta_2$ for $\Delta_1$ consisting only of positive cells and $\Delta_2$ consisting only of negative, notice that every vertex of $\Delta$ is contained in $\Delta_1$. So, consider a vertex $v$ of $\Delta_1$ of degree at least three which is neither the initial nor terminal vertex. Then $v$ must be either be initial or terminal for some positive cell and since the top path has label $a_1\ldots a_n$ it follows by a recursive argument that there exists $k$ such that the first edge in $I(v)$ is labelled $a_{k-1}$ and the first edge in $O(v)$ is labelled $a_k$. A fact we will use implicitly throughout this proof is that for a given $a_i$ there is at most one relation of the form $(a_h,a_h\ldots a_i)$ and at most one of the form $(a_j,a_i\ldots a_j)$. For each case that follows we consider what diagram could be built beginning from a vertex with one incoming edge and one outgoing edge.

    \begin{enumerate}
        \item Since $a_k$ is not a foot it must be in the support of an isolated bump and so the only relations it appears in which could add an edge incident to $v$ are $(a_{k-1},a_{k-1}a_k)$ and $(a_{k+1},a_ka_{k+1})$, and we achieve this by adding a positive cell along $a_{k-1}$ and a negative cell along $a_ka_{k+1}$. Any way of adding further edges incident to $v$ will produce dipoles.
        \item Symmetric to (2).
        \item Since $a_{k-1}$ is a destination, the relation here has the form $(a_{k-1},a_j\ldots a_{k-1})$ for some $j$ and so we may attach a positive or negative cell of this form, which introduces a second incoming edge labelled $a_{k-1}$, and now we may repeat this to add an arbitrary number of such edges. Similarly, we can add an arbitrary number of outgoing edges labelled $a_k$. In order to add edges with a different label incident to $v$, we would require a relation of the form $(a_i,a_i\ldots a_{k-1})$ or $(a_i,a_k\ldots a_i)$ for some $i$, but the first implies that $a_k$ is a destination and the second implies that $a_{k-1}$ is a source.
        \item If positive cells are to be attached so as to add edges incident to $v$ they must come before any negative cells. Since $a_{k-1}$ is a source we have a relation $(a_{k-1},a_{k-1}\ldots a_j)$ where $a_j$ is the successor of $a_{k-1}$ in $\overrightarrow{S(\mathcal{T}_m)}$; if $a_{k-1}$ is maximal then $a_j=a_{m-1}$ is a destination, where $a_m=d(\mathcal{T}_m)$. As such we may attach a positive cell at $a_{k-1}$ and then at $a_j$ and so on, as far through $\overrightarrow{S(\mathcal{T}_m)}$ as we like. If we reach the end then $a_{m-1}$ is the label of the last edge added to $I(v)$ and it is then possible to add arbitrarily many edges labelled $a_{m-1}$ to $I(v)$ by attaching either that many positive cells or that many negative cells at $a_{m-1}$ (it cannot be a combination of positive and negative since they would cancel as dipoles). Whether we add those cells or not, we may then attach a sequence negative cells to add incoming edges to $v$ labelled by an interval of $\overleftarrow{S(\mathcal{T}_m)}$, depending where on $\mathcal{T}_m$ we started and where we stop (if we did not reach the end of $\mathcal{T}_m$ or did not add edges labelled $a_{m-1}$ to $I(v)$ this is still possible since there may be a cell added along the bottom path of any positive cell with top path from $S(\mathcal{T}_m)$ which does not add an edge to $I(v)$ whence attaching the inverse of such a positive cell would not form a dipole).

        In order for this diagram to be a $(a_1\ldots a_n,a_1\ldots a_n)$-diagram the last edge of $O(v)$ must be labelled $a_{l+1}$ if the last edge of $I(v)$ is labelled $a_l$ - if this pair of edges doesn't form part of the bottom path of the diagram then there must be a cell that contains it as a subpath of its top path. We look at each possibility for $a_l$ mentioned in the previous paragraph and consider the cells which must therefore be added. First notice that if $a_k$ is a source we have a relation of the form $(a_k,a_k\ldots a_j)$ for some $j$ we may add arbitrarily many positive cells or negative cells of this form. Suppose that $a_l\in S(\mathcal{T}_m)$ and this edge is in the bottom path of a positive cell, then this cell is of the form $(a_s,a_s\ldots a_l)$ where $a_s$ is the predecessor of $a_m$ in $\overrightarrow{S(\mathcal{T}_m)}$ and we therefore have a relation of the form $(a_{l+1},a_{s+1}\ldots a_{l+1})$ - indeed, we have the relation $(a_{s+1},a_{t+1}\ldots a_{s+1})$ wherever $a_t$ is the predecessor of $a_s$ in $\overrightarrow{S(\mathcal{T}_m)}$. If the edge $a_{l+1}$ was in the bottom path of a positive cell then it must have the form $(a_j, a_{m+1}\ldots a_j)$ for some $j>k$ and $a_j$ would be the label of the predecessor in $O(v)$ and must also be in the bottom path of a positive cell of the form $(a_{j'}, a_j\ldots a_{j'})$ for $j'>j>k$. Continuing this argument indefinitely we can see that this is not possible since the first element of $O(v)$ is $a_k$ and, so, the edge labelled $a_{l+1}$ must be the bottom path of a negative cell and it must have the form $(a_{l+1},a_{s+1}\ldots a_{l+1})$. We can now apply the same argument to $a_{s+1}$ to see it must be the bottom path of a negative cell of the form $(a_{s+1},a_{t+1}\ldots a_{s+1})$ and so on until we reach the edge labelled $a_k$. Noticing that $a_{s+1}\in D(\mathcal{T}_m)$ when $a_s\in S(\mathcal{T}_m)$ and $a_s\neq s(\mathcal{T}_m)$ completes the proof for this case. If $a_l=a_{m-1}$ where $a_m=d(\mathcal{T}_m)$ then $a_{l+1}=d(\mathcal{T}_m)$ and the argument follows the same way. If $a_l\in S(\mathcal{T}_m)$ and this edge is the bottom path of a negative cell such that $m\geq k-1$ then the argument is the same; otherwise (i.e. if $m<k-1$) a symmetric argument shows that we must add a sequence of positive cells which add edges to $O(v)$ labelled by an interval from $\overleftarrow{D(\mathcal{T}_m)}$ in order to get the final edge to have label $a_{l+1}$. Finally, if $a_l=a_m=s(\mathcal{T}_m)$ then $a_{l+1}$ is a source and we may attach arbitrarily many cells of the form $(a_{l+1},a_{l+1}\ldots a_j)$ for some $j$ (either all positive or all negative).
        \item Symmetric to (4)
        \end{enumerate}
\end{proof}

\begin{prop}
    The map $\delta:\langle B\rangle\rightarrow D(\mathcal{P},a_1\ldots a_n)$ is an isomorphism.
\end{prop}

\begin{proof}
    Let $g=b_{i_1}^{\epsilon_1}\ldots b_{i_m}^{\epsilon_m}\in\langle B\rangle\setminus\{1\}$ for $\epsilon_i=\pm1$ be a freely reduced word. Then consider $\Delta=g\delta=\beta_{i_1}^{\epsilon_1}\circ\ldots\circ\beta_{i_m}^{\epsilon_m}$ and choose one of its dipoles. This dipole consists of two cells $\pi_j,\pi_k$ such that $\pi_j$ is a cell of a factor $\beta_{i_j}^{\epsilon_{j}}$ and $\pi_k$ is a cell of a factor $\beta_{i_k}^{\epsilon_{k}}$ for some $j<k$. Since each relation of $\mathcal{P}$ occurs in precisely one of the diagrams $\beta_i$ we must have $\beta_{i_j}=\beta_{i_k}$ and $\epsilon_{j}=-\epsilon_{k}$.
    
    Our aim is to reduce $\Delta$ and see that it is non-trivial. Consider a factor $\beta_{i_k}^{\epsilon_k}$ in the unreduced product above. If either of its cells remain after reducing then we're done, so suppose both are reduced. Notice that if the two dipoles are both formed with the cells of a single factor equal to $\beta_{i_k}^{-\epsilon_k}$ then there is a word equivalent to $g$ containing $b_{i_j}^{\epsilon_j}b_{i_k}^{\epsilon_k}$ as a subword and this forms a cancellable pair. Thus there must exist a factor where this does not happen; otherwise, $g$ would be equivalent to the empty word. So let $\beta_{i_k}^{\epsilon_k}$ be such a factor; there must be two distinct factors equal to $\beta_{i_k}^{-\epsilon_k}$ which each reduce with one of the cells of $\beta_{i_k}^{\epsilon_k}$ and we are thus left with a cell from each factor $\beta_{i_k}^{-\epsilon_k}$. These remaining cells cannot reduce with each other since they are labelled by different relations. Since the factor we chose with this property was arbitrary we can see there must be at two least cells remaining once $\Delta$ is reduced. Thus $\ker\delta$ is trivial and $\delta$ is injective. 

    Now, let $\Delta$ be a reduced $(a_1\ldots a_n,a_1\ldots a_n)$-diagram over $\mathcal{P}$ - we want to show that a diagram equivalent to $\Delta$ can be decomposed as a product of generator diagrams. Consider the cells $\pi_1,\ldots,\pi_m$ of $\Delta$ such that $\lfloor\pi_l\rfloor$ is a subpath of $\lfloor\Delta\rfloor$ in ascending order.

    First notice that, for a given $\pi_l$, if there does not exist a cell $\pi$ such that $\pi_l$ and $\pi$ form a generator diagram in $\Delta$ then $\pi_l$ must be negative. To see this, suppose some $\pi_l$ is a positive cell of the form $(a_i,a_i\ldots a_{j-1})$ without loss and consider its terminal vertex $v$. We can then see from Lemma~\ref{lemma:vertex} that $v$ must be the initial vertex of a negative cell of the form $(a_{i+1}\ldots a_j,a_j)$ with its bottom path being a subpath of $\lfloor\Delta\rfloor$, which means this cell must be $\pi_{l+1}$ and thus $\pi_l$ and $\pi_{l+1}$ form a generator diagram $\beta$ such that $\Delta=\Delta'\circ\beta$.

    Now suppose $\pi_l$ is a cell for which there does not exist a cell $\pi$ such that $\pi_l$ and $\pi$ form a generator diagram in $\Delta$; as noted, it must be negative. Suppose without loss that it has the form $(a_i\ldots a_{j-1},a_i)$ consider its terminal vertex $v$. By Lemma~\ref{lemma:vertex} we can see the label of the first edge in $I(v)$ must be a source and the label of the first edge in $O(v)$ must be a destination, and since $\pi_l$ is negative there, then, must be a positive cell $\pi$ of the form $(a_i,a_i\ldots a_{j-1})$ whose terminal vertex is $v$. We can now introduce a dipole $(\pi',\pi'')$ along the subpath of the bottom path of $\Delta$ labelled $a_{i+1},\ldots,a_j$ to obtain an equivalent diagram $\Bar{\Delta}$ so that $\pi$ and $\pi'$ form a generator diagram while $\pi_l$ and $\pi''$ form a generator diagram $\beta$ such that $\Bar{\Delta}=\Bar{\Delta}'\circ\beta$, and notice $\Bar{\Delta}'$ is reduced.

    Now, as we have seen, if any of the cells $\pi_1,\ldots,\pi_m$ are positive then we can decompose $\Delta=\Delta'\circ\beta$ and $\Delta'$ has fewer cells than $\Delta$. If they are all negative then by Lemma~\ref{lemma:vertex} at least one of them must not be part of a generator diagram in $\Delta$ whence we obtain an equivalent diagram $\Bar{\Delta}$ which decomposes $\Bar{\Delta}=\Bar{\Delta}'\circ\beta$ such that $\Bar{\Delta}'$ has the same number of cells as $\Delta$ but has fewer cells which do not immediately form a generator diagram. This proves the claim and it follows that $\delta$ is surjective.
\end{proof}

\noindent
We have now proved Theorem~\ref{thm:diagram}. A more detailed statement is given below.

\begin{theorem}\label{thm:detailed}
    Let $B=\{b_1,\ldots,b_n\}$ be a geometrically fast set of bumps, let $k$ be the number of isolated bumps in $B$ and let $A_1,\ldots,A_{2n+k}$ be the canonical partition of the support of $B$. Then the fast bump group $\langle B\rangle$ is isomorphic to the diagram group $D(\mathcal{P},A_1\ldots A_{2n+k})$ over the presentation
    \[
    \mathcal{P}=\langle A_1,\ldots,A_{2n+k}\ |\ A_i=A_iA_{i+1}\ldots A_{j-1}, A_j=A_{i+1}\ldots A_{j-1}A_j \text{ for each } b\in B\rangle
    \] where $A_i$ and $A_j$ is the source and destination of each $b\in B$ respectively.
\end{theorem}

\section{$PF_4$ is isomorphic to $F_4$}\label{sec:IsomorphismF4PF4}

We now turn to the $n=4$ case of the isomorphism type question for these groups. Consider a geometrically fast set $B=\{b_1,b_2,b_3,b_4\}$. By Theorem~\ref{thm:PingPong} the group $\langle B\rangle$ is invariant under isomorphism of its dynamical diagram so we may study isomorphism types of geometrically fast bump groups via these objects. Dynamical diagrams are combinatorially nice and easily enumerable; there are $\prod_{i=0}^{n-1}(2i+1)=105$ distinct dynamical diagrams of $n=4$ bumps, and so they generate at most $105$ distinct isomorphism classes of groups.

We can, however, do much better than this since many of these diagrams are decomposable into dynamical diagrams with fewer bumps in a way which respects the isomorphism type of the groups they generate. For example, if $B_1,B_2$ are geometrically fast sets with disjoint support then the dynamical diagram for $B_1\sqcup B_2$ decomposes naturally into the dynamical diagrams for $B_1$ and $B_2$, while $\langle B_1\sqcup B_2\rangle\cong\langle B_1\rangle\times\langle B_2\rangle$. Define an auxiliary graph $G$ to have vertex set $B$ and an edge $\{b_1,b_2\}$ wherever $\supt(b_1)\cap\supt(b_2)\neq\emptyset$ but $\supt(b_1)\not\subseteq\supt(b_2)$ and $\supt(b_2)\not\subseteq\supt(b_1)$. Then we say the dynamical diagram of $B$ is \textbf{irreducible} if $G$ is connected. It can be shown that there are precisely 27 irreducible dynamical diagrams with four edges.

Among these 27 diagrams we can reduce the problem even further. It is well-known that for a group $G$ acting on a set we have $\supt(g^h)=\supt(g)h$. Since $\langle b_1,b_2,b_3,b_4\rangle\cong\langle b_1^{b_2},b_2,b_3,b_4\rangle\cong\langle b_1^{b_3},b_2^{b_4^{-1}},b_3,b_4\rangle\cong\ldots$ and so on, we can use this fact to witness that many distinct dynamical diagrams generate isomorphic groups (so long as the set obtained is still fast). Of the 27 irreducible dynamical diagrams it can be shown that

\begin{itemize}
    \item 17 generate groups isomorphic to $F_4$;
    \item 10 generate another single isomorphism type.
\end{itemize}

\noindent
This unknown isomorphism type has been dubbed pseudo-$F_4$, sometimes called $PF_4$. One of the 10 dynamical diagrams generating $PF_4$ is shown in Figure~\ref{fig:FourBumps}. In what follows we will show that $PF_4\cong F_4$.

\subsection{$F_4$ as a diagram group}

\subsubsection{Standard}

Recall that a \textbf{standard $4$-adic interval} in $[0,1]$ is any  interval of the form
\[
\biggl[\frac{i}{4^n},\frac{i+1}{4^n}\biggr]
\]
where $0\leq i<4^n$.  The standard $4$-adic intervals in $[0,1]$ form a rooted $4$-ary tree under inclusion. There is a natural partial action of $F_4$ on these intervals, where $f\in F_4$ maps an interval $J$ to an interval $J'$ if $f$ is linear on $J$ and $f(J)=J'$. Thus we may represent the elements of $F_4$ using pairs of $4$-ary rooted trees (just as $F$ can be represented using pairs of binary rooted trees \cite{cfp}).

This is closely related to the standard representation of $F_4$ as a diagram group. Guba and Sapir show in \cite{diagrams} that $F_4$ is isomorphic to the diagram group over $\langle x\ |\ x^4=x\rangle$ with base $x$ (in fact any base word will do). As discussed in \cite[p.~1111]{subgroups} any reduced diagram $\Delta$ over this presentation (since it is tree-like) can be decomposed as $\Delta=\Delta_1\circ\Delta_2$ where $\Delta_1$ only contains cells of the form $(x,x^4)$ and $\Delta_2$ only contains cells of the form $(x^4,x)$. If $f\in F_4$ is represented by a reduced diagram $\Delta$ then the strand diagrams of $\Delta_1$ and $\Delta_2^{-1}$ (omitting the top boundary vertex of each) give a pair of binary rooted trees representing $f$. In the other direction, given a pair of binary trees we can horizontally reflect the codomain tree and join the two along their leaves according to the action of $f$ to obtain the interior of a strand diagram where every edge is unlabelled, which is equivalent to being labelled by the same letter $x$, and the diagram of this represents $f$.

\subsubsection{Intervals Types}

The \textbf{type} of a standard $4$-adic interval is the value of $i$ modulo~$3$ (either $0$, $1$, or~$2$).  It is easy to check that if a $4$-adic interval has type $\tau$, then its four children have types $\tau$, $\tau+1$, $\tau+2$, and $\tau$, respectively, where the addition is modulo~$3$.

Note that if $T$ is any finite, rooted subtree of the tree of $4$-adic intervals, then the types of the leaves of $T$ are precisely $0,1,2,0,1,2,0,1,2,\ldots$, with the last leaf having type~$0$.  It follows that if $f\in F_4$ has domain partition $J_1,\ldots, J_n$ and range partition $J_1',\ldots,J_n'$, then each pair of intervals $J_i,J_i'$ have the same type.  That is, the partial action of $F_4$ on the set of standard $4$-adic intervals is type-preserving.

Given this, we may use the connection between pairs of rooted trees and diagrams to find a new semigroup presentation over which $F_4$ is a diagram group. Indeed, given a pair of $4$-ary rooted trees representing an element $f$ we can form the interior of a strand diagram representing $f$ as before, except that instead of having unlabelled edges we label each according to the type of the interval which that edge led to in the rooted tree. 

It follows from this analysis that $F_4$ is isomorphic to the diagram group over the presentation
\[
\langle x_0,x_1,x_2 \mid x_0=x_0x_1x_2x_0,\;x_1=x_1x_2x_0x_1,\;x_2=x_2x_0x_1x_2\rangle
\]
with base word $x_0$.  Indeed, any base word of the form $(x_0x_1x_2)^n x_0$ gives~$F_4$.

\subsubsection{Even More Letters}
One way of looking at the introduction  of the letters $x_0$, $x_1$, and $x_2$ above is that they represent orbits of intervals under the partial action  of $F_4$ on the standard $4$-adic intervals.  Of course, this isn't quite true---although the intervals of type $1$ form a single orbit and the intervals of type $2$ form a single orbit, there are actually \textit{four} orbits of intervals of type~0.  Specifically, the interval $[0,1]$ of type~$0$ is in its own orbit, the remaining intervals that contain~$0$ form an orbit, the remaining intervals that contain~$1$ form an orbit, and the last orbit consists of all intervals of type~$0$ that lie in~$(0,1)$.

This leads to another diagram group representation of $F_4$, where we break the letter $x_0$ into four letters $u_0,v_0,w_0,x_0$ such that the letter $u_0$ represents the whole interval~$[0,1]$, the letter $v_0$ represents other intervals that contain~$0$, the letter $w_0$ represents other intervals that contain~$1$, and the letter $x_0$ represents intervals of type~$0$ that lie in $(0,1)$.  The resulting semigroup presentation is
\[
\bigg\langle u_0,v_0,w_0,x_0,x_1,x_2 \;\biggl|\; \begin{array}{l}u_0=v_0x_1x_2w_0,\; v_0=v_0x_1x_2x_0,\; w_0=x_0x_1x_2w_0, \\[2pt] x_0=x_0x_1x_2x_0,\; x_1=x_1x_2x_0x_1,\; x_2=x_2x_0x_1x_2
\end{array}\bigg\rangle
\]
It is not hard to show that the diagram group over this presentation with base word~$u_0$ is isomorphic to~$F_4$.  The base word $v_0x_1x_2w_0$ also suffices, in which case we can remove the generator $u_0$ and the relation \mbox{$u_0=v_0x_1x_2w_0$} from the presentation without affecting the isomorphism type by Theorem~\ref{thm:moves}.

\subsection{The Proof}

We now prove that pseudo-$F_4$ is isomorphic to $F_4$.  First we label the eight feet of the bumps using the eight letters $A$, $B$, $C$, $D$, $\overline{A}$, $\overline{B}$, $\overline{C}$, and $\overline{D}$, as shown in Figure~\ref{fig:FourBumpsLabeled}.  According to Theorem~\ref{thm:detailed}, pseudo-$F_4$ is isomorphic to the diagram group for the presentation
\[
\bigg\langle \begin{array}{l}A,B,C,D,\\[2pt] \overline{A},\overline{B},\overline{C},\overline{D}\end{array} \;\biggl|\; \begin{array}{l}A=ABCD,\; B=BCD\overline{DC},\; C=CD\overline{D}, \;D=D\overline{DCB} \\[2pt] \overline{A}=\overline{DCBA},\; \overline{B}=CD\overline{DCB},\; \overline{C}=D\overline{DC},\;\overline{D}=BCD\overline{D}
\end{array}\bigg\rangle
\]
\begin{figure}
    \centering
    \includegraphics{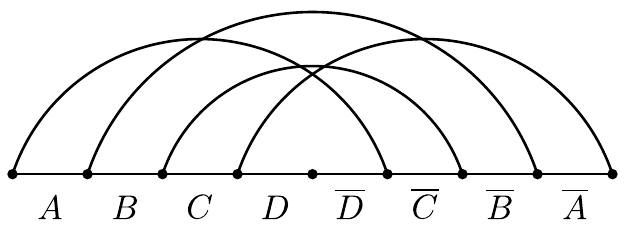}
    \caption{Labels for the eight feet of the bumps that generate pseudo-$F_4$.}
    \label{fig:FourBumpsLabeled}
\end{figure}%
with base word $ABCD\overline{DCBA}$.  Note that this presentation is symmetric with respect to the operation of: \begin{enumerate}
    \item Switching the pairs $(A,\overline{A})$, $(B,\overline{B})$, $(C,\overline{C})$, and $(D,\overline{D})$ and then\smallskip
    \item Reversing the order of the generators in every word.
\end{enumerate}
We now apply a sequence of Guba and Sapir's moves from Theorem~\ref{thm:moves} to this presentation to obtain further diagram groups that are isomorphic to pseudo-$F_4$.  First, since the last relation is $\mbox{$\overline{D}=BCD\overline{D}$}$, we can replace the $BCD\overline{D}$ in the relation $B=BCD\overline{DC}$ by a~$\overline{D}$, and similarly we can replace the $D\overline{DCB}$ in the relation $\overline{B}=CD\overline{DCB}$ by a~$D$. This yields the semigroup presentation
\[
\bigg\langle \begin{array}{l}A,B,C,D,\\[2pt] \overline{A},\overline{B},\overline{C},\overline{D}\end{array} \;\biggl|\; \begin{array}{l}A=ABCD,\; B=\overline{DC},\; C=CD\overline{D}, \;D=D\overline{DCB} \\[2pt] \overline{A}=\overline{DCBA},\; \overline{B}=CD,\; \overline{C}=D\overline{DC},\;\overline{D}=BCD\overline{D}
\end{array}\bigg\rangle
\]
Next we replace $D\overline{DC}$ by $\overline{C}$ in the relation $D=D\overline{DCB}$, and we replace $CD\overline{D}$ by $C$ in the relation $\overline{D}=BCD\overline{D}$, yielding the presentation
\[
\bigg\langle \begin{array}{l}A,B,C,D,\\[2pt] \overline{A},\overline{B},\overline{C},\overline{D}\end{array} \;\biggl|\; \begin{array}{l}A=ABCD,\; B=\overline{DC},\; C=CD\overline{D}, \;D=\overline{CB} \\[2pt] \overline{A}=\overline{DCBA},\; \overline{B}=CD,\; \overline{C}=D\overline{DC},\;\overline{D}=BC
\end{array}\bigg\rangle
\]
Next we replace $CD$ by $\overline{B}$ in the relation $C=CD\overline{D}$, and we replace $\overline{DC}$ by $B$ in the relation $\overline{C}=D\overline{DC}$, yielding the presentation
\[
\bigg\langle \begin{array}{l}A,B,C,D,\\[2pt] \overline{A},\overline{B},\overline{C},\overline{D}\end{array} \;\biggl|\; \begin{array}{l}A=ABCD,\; B=\overline{DC},\; C=\overline{BD}, \;D=\overline{CB} \\[2pt] \overline{A}=\overline{DCBA},\; \overline{B}=CD,\; \overline{C}=DB,\;\overline{D}=BC
\end{array}\bigg\rangle.
\]
By Guba and Sapir's theorem, the diagram group over this presentation with base word $ABCD\overline{DCBA}$ remains isomorphic to pseudo-$F_4$.

Next we break the symmetry by substituting $CD$ for $\overline{B}$, substituting $DB$ for $\overline{C}$, and substituting $BC$ for $\overline{D}$ in all of the other relations as well as in the base word.  This yields the semigroup presentation
\[
\bigg\langle \begin{array}{l}A,B,C,D,\\[2pt] \overline{A},\overline{B},\overline{C},\overline{D}\end{array} \;\biggl|\; \begin{array}{l}A=ABCD,\; B=BCDB,\; C=CDBC, \;D=DBCD \\[2pt] \overline{A}=BCDBCD\overline{A},\; \overline{B}=CD,\; \overline{C}=DB,\;\overline{D}=BC
\end{array}\bigg\rangle
\]
with base word $ABCDBCDBCD\overline{A}$.  We can now eliminate the generators $\overline{B}$, $\overline{C}$, and $\overline{D}$ as well as the corresponding relations to get the semigroup presentation
\[
\bigg\langle A,B,C,D,\overline{A} \;\biggl|\; \begin{array}{l}A=ABCD,\; B=BCDB,\; C=CDBC, \;D=DBCD \\[2pt] \overline{A}=BCDBCD\overline{A}
\end{array}\bigg\rangle
\]
with base word $ABCDBCDBCD\overline{A}$.  If we replace $BCDB$ by $B$ in the relation for $\overline{A}$ as well as twice successively in the base word, we obtain the semigroup presentation
\[
\bigg\langle A,B,C,D,\overline{A} \;\biggl|\; \begin{array}{l}A=ABCD,\; B=BCDB,\; C=CDBC, \;D=DBCD \\[2pt] \overline{A}=BCD\overline{A}
\end{array}\bigg\rangle
\]
with base word $ABCD\overline{A}$, and the corresponding diagram group remains isomorphic to pseudo-$F_4$. 

This is almost the same as the diagram group representation of $F_4$ discussed above.  To finish, we introduce a new generator $E$ to the presentation:
\[
\bigg\langle A,B,C,D,\overline{A},E \;\biggl|\; \begin{array}{l}A=ABCD,\; B=BCDB,\; C=CDBC, \;D=DBCD \\[2pt] \overline{A}=BCD\overline{A},\; E=D\overline{A}
\end{array}\bigg\rangle
\]
Next we substitute $E$ for $D\overline{A}$ in the relation for $\overline{A}$:
\[
\bigg\langle A,B,C,D,\overline{A},E \;\biggl|\; \begin{array}{l}A=ABCD,\; B=BCDB,\; C=CDBC, \;D=DBCD \\[2pt] \overline{A}=BCE,\; E=D\overline{A}
\end{array}\bigg\rangle
\]
and similarly in the base word to get the new base word $ABCE$.
Next we substitute $BCE$ for $\overline{A}$ in the relation for~$E$:
\[
\bigg\langle A,B,C,D,\overline{A},E \;\biggl|\; \begin{array}{l}A=ABCD,\; B=BCDB,\; C=CDBC, \;D=DBCD \\[2pt] \overline{A}=BCE,\; E=DBCE
\end{array}\bigg\rangle
\]
which allows use to eliminate the generator $\overline{A}$ and its corresponding relation.  The result is the presentation
\[
\bigg\langle A,B,C,D,E \;\biggl|\; \begin{array}{l}A=ABCD,\; B=BCDB,\; C=CDBC, \;D=DBCD \\[2pt] E=DBCE
\end{array}\bigg\rangle
\]
with base word $ABCE$.  This is exactly the diagram group representation for $F_4$ discussed above, with $v_0=A$, $w_0=E$, $x_0=D$, $x_1=B$, and $x_2=C$, so we conclude that pseudo-$F_4$ is isomorphic to~$F_4$.

\printbibliography

\end{document}